\documentclass[12 pt]{amsart}
\usepackage{fullpage, amsthm, amsmath, amsfonts, amssymb, url, mathrsfs, graphicx}
\usepackage{esint}
\usepackage{hyperref}
\usepackage{enumerate}
\newcommand{\R}{\mathbb R}

\def \colonequals {\mathrel{\mathop:}=}

 \numberwithin{equation}{section}
 \newcommand{\Lip}{\mathrm{Lip}}

\newtheorem{thm}{Theorem}[section]
\newtheorem{defin}[thm]{Definition}
\newtheorem{lem}[thm]{Lemma}
\newtheorem{cor}[thm]{Corollary}

\newtheorem*{thm*}{Theorem}
\newtheorem*{lem*}{Lemma}

\theoremstyle{remark}

\newtheorem*{ack}{Acknowledgements}

\title{Parabolic NTA domains in $\mathbb R^2$}
\date{\today}
\author{Max Engelstein}
\thanks{Partially supported by an NSF Graduate Research Fellowship, DGE 1144082 and the University of Chicago RTG grant DMS 1246999}
\address{ Department of Mathematics\\MIT\\Cambridge, MA 02139}
\email{ maxe@mit.edu}

\begin{document}

\maketitle

\begin{abstract}
 We show that each connected component of the boundary of a parabolic NTA domain in $\mathbb R^2$ is given by a graph. We then apply this observation to classify blowup solutions in $\mathbb R^2$ to a free boundary problem for caloric measure first considered by Hofmann, Lewis and Nystr\"om \cite{hlncaloricmeasure}. 
\end{abstract}

\section{Introduction} 
We study parabolic non-tangentially accessible ({\it NTA}) domains in $\mathbb R^2$ (one space and one time dimension). In particular, we show that the topology of $\mathbb R^2$, combined with the anisotropic nature of parabolic NTA domains, forces such domains to either lie above the graph of one function of time or between the graphs of two functions of time (see Theorem \ref{NTAdomainisagraphdomain} below). This fact should be contrasted with the existence of elliptic (in the sense of \cite{jerisonandkenig}) NTA domains in $\mathbb R^2$ whose boundary is not locally given by the graph of a function at any point or scale (e.g. Wolff snowflakes, see \cite{wolffsnowflakes}). 

Jerison and Kenig, in \cite{jerisonandkenig}, introduced (elliptic) NTA domains, as a generalization of Lipschitz domains in which the boundary behavior of harmonic functions can be studied. In particular, harmonic measure is doubling in NTA domains (\cite{jerisonandkenig}, Lemma 4.9) and harmonic functions satisfy a boundary Harnack inequality (\cite{jerisonandkenig}, Theorem 5.1). On the other hand, NTA domains do not necessarily have finite perimeter or boundaries which are locally given by graphs. Thus NTA domains are a natural setting in which to study solutions to elliptic PDE with minimal regularity assumptions. 

The NTA condition is also natural in the study of rough sets. For example, NTA domains satisfy the ``two-disk" condition which was used by David and Jerison in \cite{davidandjerison} to construct big pieces of Lipschitz graphs inside of sets whose surface measure is Ahlfors-regular. The precise relationship between the presence of big pieces of Lipschitz graphs, uniform rectifiability and the NTA condition has been the subject of much recent work (for a small sample see \cite{BadgerLipschitzApprox}, \cite{URandHMIII}, \cite{fiveauthorschordarc} and the references therein). 

Parabolic NTA domains, introduced by Lewis and Murray in \cite{lewisandmurray}, are less understood. The boundary behavior of caloric functions has been studied in parabolic Lipschitz domains (see \cite{fabesgarofalosalsa}, \cite{BrownParabolicAreaIntegrals} and \cite{lewisandmurray}) and parabolic Reifenberg flat domains (see \cite{hlncaloricmeasure} and \cite{engelsteincaloricmeasure}) but for arbitrary parabolic NTA domains it is unknown, for example, if caloric measure is always doubling (as the domain may fail to ``separate" space in the sense of equation (1.1) in \cite{hlncaloricmeasure}). Similarly, the relationship between the topological constraint of being parabolic NTA and uniform rectifiability (in the parabolic sense) is still being investigated (for some important recent work in this direction, see \cite{nysandstrom}).  Our Theorem \ref{NTAdomainisagraphdomain} is a step towards understanding the geometry of parabolic NTA domains.

We also give an application of Theorem \ref{NTAdomainisagraphdomain} to the regularity of the free boundary in a one-phase problem for caloric measure. In particular, Theorem \ref{NTAdomainisagraphdomain}, combined with Nystr\"om's work in \cite{nystromindiana}, implies that all ``blowup" solutions to a free boundary problem for caloric measure must be half-planes (see Section \ref{secglobalfreeboundaryproblem} for details). Free boundary problems for harmonic and caloric measure are a subject of great interest, as they are connected to questions in geometric measure theory, potential theory and complex analysis  (see \cite{toroexpository}). 

 For harmonic measure, blowup solutions in $\mathbb R^2$ were proven to be half-planes by Pommerenke \cite{pommerenke}, and, in higher dimensions, ``flat" blowups were proven to be half-planes by Alt and Caffarelli \cite{altcaf} (see also \cite{kenigtorooneeverywhere}). Without the flatness assumption, there are other blowup solutions (which are analogous to non-flat minimal surfaces) in $\mathbb R^n$ with $n \geq 3$ (see, e.g., example 2.7 in \cite{altcaf}). These non-flat solutions have been studied extensively (see \cite{weissonephase}, \cite{cafjerken}, \cite{desilvajersingular}, \cite{desilvajergraphs}, \cite{jersavinstability} and \cite{JerisonKamburov}). 

The free boundary problem for the parabolic Poisson kernel was first studied by Hofmann, Lewis and Nystr\"om, \cite{hlncaloricmeasure}, and subsequently by Nystr\"om, \cite{nystromfenn}, \cite{nystromindiana}, \cite{nystromzeitschrift}, and then by the author, \cite{engelsteincaloricmeasure}. Theorem 1.10 in \cite{engelsteincaloricmeasure} states, under {\it a priori} flatness assumptions, that the only blowup solutions to the parabolic free boundary problem are half spaces. Theorem \ref{globalblowupsinr2}, below, removes the flatness assumption in $\mathbb R^2$. 

We know very little about non-flat blowup solutions to the parabolic problem and it would be interesting to find blowup solutions which are time-dependent (i.e. not an elliptic blowup solution cross $\mathbb R$).  Theorem \ref{globalblowupsinr2} shows that non-flat blowup solutions to the parabolic problem don't exist in $\R^2$ but it is still an open question as to whether they exist in $\mathbb R^3$.  Finally, we note that the study of non-flat blowup solutions is connected to questions of existence, uniqueness and regularity for a parabolic free boundary problem arising in combustion (see, e.g., \cite{cafvaz} and \cite{AnderssonWeissDomainVariations}). 

Let us outline the structure of what follows: in Section \ref{geometryofparabolicNTAdomains} we introduce the concept of a parabolic NTA domain and prove that each connected component of the boundary of a parabolic NTA domains is given by a graph (Theorem \ref{NTAdomainisagraphdomain}). 

In Section \ref{secglobalfreeboundaryproblem} we combine our work with that of Nystr\"om, \cite{nystromindiana}, to classify blowup solutions in $\R^2$ to the parabolic free boundary problem (Theorem \ref{globalblowupsinr2}). We then briefly describe how this allows us to prove regularity, without assumed flatness, in $\R^2$ for a more general parabolic free boundary problem.

\begin{ack}  We thank Professor Marianna Cs\"ornyei, who helped us understand how complicated the boundary of a domain in $\mathbb R^2$ could be. We also acknowledge the helpful comments of an anonymous referee. Finally, we thank our advisor, Professor Carlos Kenig, who introduced us to free boundary problems and whose encouragement and support made this project possible.  \end{ack}

\section{The Geometry of Parabolic NTA Domains}\label{geometryofparabolicNTAdomains}

We work with points $(x,t) \in \mathbb R^2$ under the parabolic metric $\|(x_1, t_1) - (x_2, t_2)\| \equiv |x_1 - x_2| + |t_1-t_2|^{1/2}$. Let $$C_r(x_0,t_0) \equiv \{(x,t) \in \mathbb R^2| |x-x_0|< r, |t-t_0| < r^2\}$$ be the parabolic cylinder centered at $(x_0,t_0)\in \mathbb R^2$ for scale $r > 0$. If we have a point, $A \in \mathbb R^2$, we denote its space and time coordinates by $A_x$ and $A_t$ respectively.

Given a domain, $\Omega$, we define a ``surface ball", $$\Delta_r(x_0,t_0) \equiv C_r(x_0,t_0)\cap \partial \Omega,$$ for $(x_0,t_0) \in \partial \Omega$ and $r > 0$. We also need a parabolic analogue of surface measure, which we call $\sigma$, given (in $\mathbb R^{n+1}$) by \begin{equation}\label{defofsigma} d\sigma \equiv d\mathcal H^{n-1}(x)|_{\{(x,t)\in \partial \Omega\}} dt.\end{equation} So, in $\mathbb R^2$, $d\sigma \equiv \#\{x \mid (x,t) \in \partial \Omega\}dt$.

We can now define a parabolic {\it non-tangentially accessible} (NTA) domain (c.f. Chapter 3, Section 6 in \cite{lewisandmurray}).

\begin{defin}\label{parabolicNTAdomain}[Parabolic NTA domain]
A connected open set, $\Omega \subset \mathbb R^{n+1}$, is {\bf non-tangentially accessible} (NTA) if there are constants $\lambda \geq 2, \gamma \geq \sqrt{2\lambda}+1$ for which the following hold:
\begin{enumerate}
\item $\Omega$ satisfies the foward and backwards corkscrew condition: for any $(Q,\tau) \in \partial \Omega$ and $r > 0$ there exists $A_r^{\pm}(Q,\tau) \equiv (X^{\pm}_r(Q,\tau), t_r^{\pm}(Q,\tau)) \in \Omega \cap C_r(Q,\tau)$ such that $$\lambda^{-1}r^2 \leq \min(t^+_r - \tau, \tau-t^-_r) \leq r^2,$$$$ \mathrm{dist}(A_r^{\pm}(Q,\tau), \partial \Omega) \geq \frac{r}{\lambda}.$$
\item $\overline{\Omega}^c$ satisfies the forwards and backwards corkscrew condition.
\item $\Omega$ satisfies the Harnack chain condition: if $\varepsilon > 0, (X_1, t_1), (X_2, t_2) \in \Omega$ such that $$(t_2-t_1)^{1/2} > \gamma^{-1} \|(X_1, t_1)-(X_2, t_2)\|,$$$$\mathrm{dist}((X_i,t_i), \partial \Omega) > \varepsilon,\; i = 1,2,$$ then there is a ``Harnack chain" of overlapping cylinders between $(X_1, t_1)$ and $(X_2, t_2)$. We say that $\{C_{r_j}(Y_j,s_j)\}_{j=1}^\ell$ is a Harnack chain from $(X_1, t_1)$ to $(X_2, t_2)$ if there is a constant $c(\gamma) \geq 1$ such that \begin{subequations}
\label{harnackchain}
\begin{align}
&(X_1, t_1) \in C_{r_1}(Y_1, s_1), \; (X_2,t_2) \in C_{r_\ell}(Y_\ell, s_\ell), \label{startandstop}\\
&C_{r_{j+1}}(Y_{j+1}, s_{j+1})\cap C_{r_{j}}(Y_{j}, s_{j}) \neq \emptyset,\; j= 1,2,\ldots, \ell-1,\label{overlap}\\
&c(\gamma)^{-1} \mathrm{dist}((Y_j, s_j), \partial \Omega) \leq r_j \leq c(\gamma)\mathrm{dist}((Y_j, s_j), \partial \Omega), \; j =1,2,\ldots, \ell, \label{comp}\\
& s_{j+1} - s_j \geq c(\gamma)^{-1}r_j^2,\; i = 1,2,\ldots, \ell-1, \label{spreadintime}\\
& \ell \leq c(\gamma) \log(2+\varepsilon^{-1}\|(X_1, t_1)-(X_2, t_2)\|). \label{notsolong}
    \end{align}
\end{subequations}
\end{enumerate}
\end{defin}

We first note that other definitions of parabolic NTA domains (e.g. the one in \cite{lewisandmurray}) do not include the restriction $\gamma \geq \sqrt{2\lambda} + 1$. However, this restriction is necessary if we are to guarantee that $A^{+}_r(Q,\tau)$ and $A^-_r(Q,\tau)$ are separated enough in time to satisfy condition (3) above. Furthermore, if $\Omega$ is a Reifenberg flat domain, a graph domain or a cylinder over an (elliptic) NTA domain then we can always take $\gamma \geq \sqrt{2\lambda} + 1$.  Since these particular classes of parabolic NTA domains have been the focus of most of the prior work in this area, our restriction that $\gamma \geq \sqrt{2\lambda} + 1$ comports well with the results of others.

One can verify that for any elliptic NTA domain (in the sense of \cite{jerisonandkenig}), $\Omega$, the domain $\Omega \times \R$ is a parabolic NTA domain in the sense of Definition \ref{parabolicNTAdomain}. However, the above definition also allows for variation in time; for example, if $f\in \mathrm{Lip}(1/2)$, then $\Omega = \{(x,t)\mid x > f(t)\}$ is a parabolic NTA domain with NTA constants which depend on the Lipschitz norm of $f$. 

The Harnack chain condition should be viewed as a kind of ``quantitative path connectedness." In particular, it implies that if $(X_1,t_1), (X_2, t_2) \in \Omega$ and $(t_2-t_1)^{1/2} > \gamma^{-1} \|(X_1, t_1)-(X_2, t_2)\|$, then there exists a curve, $p: [0,1] \rightarrow \mathbb R^{n+1}$ such that $p(0) = (X_1, t_1), p(1) = (X_2, t_2)$ and the time coordinate of $p$ is monotonically increasing (i.e. $\alpha > \beta\Rightarrow (p(\alpha))_t > (p(\beta))_t$).

Finally, we reiterate that the boundary of an NTA domain need not be given locally by the graph of a function, nor have locally finite surface measure. For the elliptic definition, given in \cite{jerisonandkenig}, this is true even for NTA domains which are subsets of $\mathbb R^2$ (see \cite{wolffsnowflakes}). However, the parabolic definition privileges the time direction, and this allows us to conclude much greater structure on behalf of {\it parabolic} NTA domains in $\mathbb R^2$--in particular, Theorem \ref{NTAdomainisagraphdomain} below implies that each connected component of the boundary of a parabolic NTA domain in $\mathbb R^2$ is a graph (and, consequently, the boundary has locally finite $d\sigma$ measure). 

\begin{thm}\label{NTAdomainisagraphdomain}
Let $\Omega \subset \mathbb R^2$ be a parabolic NTA domain. Either there are two functions, $f,g:\R\rightarrow \R$, such that $\Omega = \{(x,t)\mid g(t) > x > f(t)\}$ or, after a possible reflection across the time axis, there is a single function, $f:\mathbb R\rightarrow \mathbb R$, such that $\Omega = \{(x,t) \mid x > f(t)\}$. 
\end{thm}

\begin{proof}
We will first show that each connected component of $\partial \Omega$ is given by a graph. Assume to the contrary. Then there are two points in a connected component of $\partial \Omega$ with the same time coordinate. After harmless translation and scaling we may assume that $(0,0), (1,0) \in \partial \Omega$ (and are in the same connected component of $\partial \Omega$). There are two possibilities. 

\smallskip

\noindent {\bf Case 1:} There exists a point, $(a, 0)$, with $0 < a < 1$ such that $(a,0) \in \Omega$. We say that a point $(x,t)$, with $t > 0$, is {\it forward accessible} from $(a,0)$ if there exists a curve $\gamma: [0,1] \rightarrow \mathbb R^2$ such that \begin{equation}\label{gammadef} \begin{aligned} \gamma(0) =& (a,0)\\
\gamma(1) =& (x,t)\\
\gamma(\tau) \in& \Omega,\; \forall \tau\in [0,1]\\
\gamma(\tau_1)_t >& \gamma(\tau_2)_t,\; \forall 1 \geq \tau_1 > \tau_2 \geq 0.
\end{aligned}
\end{equation} 

That is, if there is a continuous curve which moves monotonically forward in time and is contained in $\Omega$, connecting $(a,0)$ to $(x,t)$. We can similarly say that $(x,t)$, with $t < 0$ is {\it backwards accessible} from $(a,0)$ if there is a curve $\gamma$ which is continuous and contained in $\Omega$ that moves monotonically backwards in time and connects $(a,0)$ to $(x,t)$. If we don't want to specify a direction we will just say that $(x,t)$ is {\it accessible} from $(a,0)$. 

Let $\mathcal A_{(a,0)}$ denote the set of all points accessible from $(a,0)$ and let $t^+ = \sup \{t \mid (x,t) \in \mathcal A_{(a,0)}\}$ and $t^- = \inf \{t \mid (x,t) \in \mathcal A_{(a,0)}\}$. If $t^+ = +\infty$ and $t^- = -\infty$, then $(0,0)$ and $(1,0)$ are in two different connected components of $\partial \Omega$.

So we may assume that either $t^+$ or $t^-$ is finite. Without loss of generality assume that $t^+ \equiv T < \infty$. For any $\varepsilon > 0$, there exists a point, $(x_\varepsilon, T-\varepsilon) \in \Omega$, which is forward accessible from $(a,0)$. By the definition of supremum, there must be some $\tilde{T} \in (T-\varepsilon, T]$ such that $(x_\varepsilon, \tilde{T}) \in \partial \Omega$. Denote the interior forward corkscrew point at $(x_\varepsilon, \tilde{T})$ at scale $r > 0$ by $A_r^+(x_\varepsilon, \tilde{T}) \equiv  (y_r^+, t_r^+)$. That $|t_r^+ - (T-\varepsilon)|^{1/2} \geq \gamma^{-1}\|(x_\varepsilon, T-\varepsilon) - (y^+_r,t_r^+)\|$ follows from the fact that $|t_r^+ - (T-\varepsilon)|^{1/2} \geq \frac{r}{\sqrt{\lambda}}$ and $\gamma \geq \sqrt{2\lambda} + 1$. Since the two points are sufficiently separated in time, there must be a Harnack chain connecting $(x-\varepsilon, T-\varepsilon)$ and $(y_r^+, t_r^+)$. 

As mentioned above, the existence of a Harnack chain implies that there is a continuous curve $\tilde{\gamma} \subset \Omega$ which moves monotonically forward in time that connects $(x-\varepsilon, T-\varepsilon)$ and $(y_r^+, t_r^+)$. Concatenating this curve with the curve that moves monotonically forward in time and connects $(a,0)$ with $(x-\varepsilon, T-\varepsilon)$  implies that $(y_r^+, t_r^+)$ is forward accessible from $(a,0)$. As $t_r^+ > T$ for $r$ sufficiently larger than $\varepsilon$, this a contradiction and we are done. 

\begin{figure}[h]
\begin{center}\includegraphics[height=1.75in]{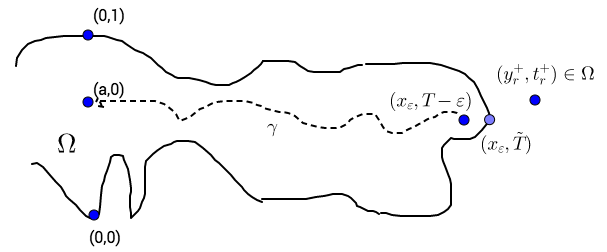}\end{center}
\caption{A contradiction occurs if $(a,0)$ cannot be connected to points with arbitrarily large time coordinates by paths increasing monotonically in time.}
\end{figure}

\medskip

\noindent {\bf Case 2:} There is a point, $(a,0) \in \Omega^c$, with $0 < a < 1$. 

Let $(X^\pm_1,t^\pm_1) \equiv (X_r^\pm(1,0), t_r^\pm(1,0))\in \Omega$ be the forward/backward interior corkscrew point at $(1,0) \in \partial \Omega$ for scale $r$ (which we will choose, large, later). Also define $(X^\pm_2, t^\pm_2) \equiv(X_{\rho}^\pm(0,0), t_{\rho}^\pm(0,0)) \in \Omega$ to be the forward/backward interior corkscrew point at $(0,0)$ for scale $\rho << a/2$ and $(X^\pm_3, t^\pm_3) \equiv (X_{\rho}^\pm(1,0), t_{\rho}^\pm(1,0)) \in \Omega$ be the forward/backward interior corkscrew point at $(1,0)$ for scale $\rho << (1-a)/2$ (see Figure \ref{figcase2}). 

We can connect $(X^+_2, t^+_2)$ to $(X^-_2, t^-_2)$ by a Harnack chain and, consequently, by a curve $\gamma_{\mathrm{bottom}}$ which stays entirely inside of $\Omega$ and is monotone in time. Furthermore, letting $\rho << 1$, we can guarantee that $\gamma_{\mathrm{bottom}}$ is short enough (using \eqref{comp} and \eqref{notsolong}) such that $a$ is greater than the space coordinate of any point on $\gamma_{\mathrm{bottom}}$. Similarly, we can construct $\gamma_{\mathrm{top}}$ between $(X^+_3, t^+_3)$ and $(X^-_3, t^-_3)$ and let $\rho$ be small enough such that $a$ is smaller than the space coordinate of any point on $\gamma_{\mathrm{top}}$. 

We claim that there exist Harnack chains connecting $(X_2^+, t^+_2)$ and $(X^+_3, t_3^+)$ to $(X^+_1, t^+_1)$. Indeed, $$\begin{aligned}(t_1^+-t_2^+)^{1/2} \geq \gamma^{-1}(|X_1^+ - X_2^+| + |t_1^+-t_2^+|^{1/2}) &\Leftrightarrow (t_1^+-t_2^+)^{1/2}(1-\gamma^{-1}) \geq \gamma^{-1}|X_1^+ - X_2^+| \\ &\Leftarrow (\frac{r^2}{\lambda} - \rho^2)^{1/2}(1-\gamma^{-1}) \geq \frac{r+\rho+1}{\gamma}\\ &\stackrel{r >> \rho}{\Leftarrow} \frac{\sqrt{3}r}{2\sqrt{\lambda}}(\gamma-1)\geq r+\rho+1\\ &\stackrel{\gamma \geq \sqrt{2\lambda}+1}{\Leftarrow} \sqrt{\frac{3}{2}}r \geq r + \rho + 1.\end{aligned}$$ Since this last equation is true if $r$ is large enough (compared to $\rho$ and $1$) we can connect $(X^+_1, t^+_1)$ and $(X_2^+, t^+_2)$ by a Harnack chain. A similar computation allows us to connect $(X^+_1, t^+_1)$ to $(X^+_3, t^+_3)$ and connect $(X^-_1, t^-_1)$ to $(X^-_2, t^-_2)$ and $(X^-_3, t^-_3)$.

So there are curves, $\gamma^\pm_2, \gamma^\pm_3$, which lie inside of $\Omega$ and have monotone increasing time coordinates which connect $(X_2^{\pm}, t^\pm_2)$ and $(X_3^{\pm}, t^\pm_3)$ respectively to $(X_1^\pm, t_1^\pm)$. The union of the six curves, $\gamma^\pm_2, \gamma^\pm_3, \gamma_{\mathrm{top}}$ and $\gamma_{\mathrm{bottom}}$, forms a simple closed curve, $C$, in $\mathbb R^2$ with the point $(a,0)$ in the interior of the bounded component of $\mathbb R^2\backslash C$. 

Since, $(a,0) \in \Omega^c$, the intersection between $\Omega^c$ and the bounded component of $\mathbb R^2\backslash C$ must be non-empty. Call this intersection, $\Gamma$.  Let $(x_0,t_0) \in \Gamma$ be a point with smallest possible time coordinate in $\Gamma$ (which exists as $\Gamma$ is closed and bounded). It is clear that $(x_0,t_0) \in \partial \Omega$ and that, for small $r$, there cannot exist a backwards in time exterior corkscrew point at $(x_0, t_0)$ for scale $r$ (as such a point would have to be contained in $\Gamma$ but have a smaller time coordinate than $(x_0,t_0)$). This is a contradiction, and so {\bf Case 2} cannot occur. 

\begin{figure}[h]\label{figcase2}
\begin{center}\includegraphics[height=1.75in]{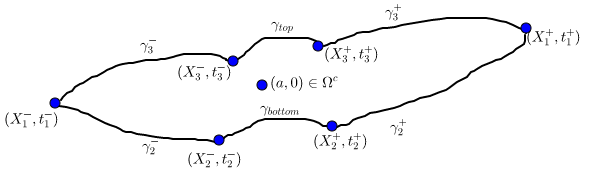}\end{center}
\caption{The 6 curves, $\gamma_2^\pm, \gamma_3^\pm, \gamma_{\mathrm{top}}$ and $\gamma_{\mathrm{bottom}}$, bound a component of $\Omega^c$ which leads to a contradiction.}
\end{figure}

\medskip

This proves that every connected component of $\partial \Omega$ is the graph of a function in time. Since $\Omega$ is connected, there are only two possibilities: either $\Omega$ is a graph domain, i.e. $\Omega = \{(x,t) \mid x > f(t)\}$, or $\Omega$ is a ``slab domain", i.e. $\Omega = \{(x,t) \mid g(t) > x > f(t)\}$. 
 \end{proof}

\subsection{Parabolic Chord Arc Domains} Dahlberg, \cite{DahlbergHM}, proved that harmonic measure and surface measure are mutually absolutely continuous in a Lipschitz domain. Mutual absolute continuity is important as it is closely related to the solvability of the Dirichlet problem for rough data.

In the parabolic setting the situation is more complicated; it was shown by Kaufman and Wu, \cite{kaufmanandwu}, that caloric measure is not necessarily mutually absolutely continuous with respect to $\sigma$, as defined in \eqref{defofsigma}, in a $\Lip(1, 1/2)$ domain. Later, Lewis and Murray, \cite{lewisandmurray}, showed that these measure are mutually absolutely continuous in a domain $\{(x_1,\ldots,x_n, t)\mid x_n \geq f(x_1,\ldots, x_{n-1},t)\}$ as long as $f\in \Lip(1,1/2)$ {\bf and} $D_t^{1/2}f \in \mathrm{BMO}(\R^n)$ (where the fractional derivative is defined in an integral sense). We call these domains {\it parabolic regular graph domains}.  If $\partial \Omega$ has two connected components both of which are graphs of functions in $\Lip(1,1/2)$ with $1/2$-time derivatives in $\mathrm{BMO}$, we will call $\Omega$ a {\it parabolic regular slab domain}. 

 Hofmann, Lewis and Nystr\"om introduced the more general {\it parabolic chord arc domains} (see \cite{hlncaloricmeasure}) and showed that, with {\it a priori} flatness assumptions, a parabolic chord arc domain contains big pieces of parabolic regular graph domains (see \cite{hlnbigpieces}). In this subsection, we prove that, in $\mathbb R^2$, a parabolic chord arc domain is in fact a parabolic regular graph domain or slab domain. 

To define parabolic chord arc domains we first recall the definition of Ahlfors regularity:

\begin{defin}\label{ahlforsregularity}
$\partial \Omega$ is {\bf Ahflors regular} if there exists an $M \geq 1$ such that for all $(Q,\tau) \in \partial \Omega$ and $R > 0$ we have $$\left(\frac{R}{2}\right)^{n+1} \leq \sigma(C_R(Q,\tau)\cap \partial \Omega) \leq MR^{n+1}.$$
\end{defin}

Following \cite{hlncaloricmeasure}, define, for $r > 0$ and $(Q,\tau) \in \partial \Omega$, \begin{equation}\label{gammainuniformrectifiability}
\gamma(Q,\tau, r) = \inf_P \left(r^{-n-3}\int_{\partial \Omega \cap C_r(Q,\tau)} d((X,t), P)^2 d\sigma(X,t)\right)
\end{equation} where the infimum is taken over all $n$-planes containing a line parallel to the $t$-axis and going through $(Q,\tau)$. This is an $L^2$ analogue of Jones' $\beta$-numbers (\cite{jonesbeta}).  We want to measure how $\gamma$, ``on average", grows in $r$ and, to that end, introduce \begin{equation}\label{whatisnu} d\nu(Q,\tau, r) = \gamma(Q,\tau, r)d\sigma(Q,\tau)r^{-1}dr.\end{equation} Recall that $\mu$ is a Carleson measure with norm $\|\mu\|_+$ if \begin{equation}\label{carlesonmeasure}
\sup_{R > 0} \sup_{(Q,\tau) \in \partial \Omega} \mu((C_R(Q,\tau) \cap \partial \Omega) \times [0,R]) \leq \|\mu\|_+ R^{n+1}.
\end{equation}

In analogy to David and Semmes \cite{DavidandSemmes} (who defined uniformly rectifiable domains in the isotropic setting) we define a parabolic uniformly rectifiable domain;

\begin{defin}\label{uniformlyrectifiable}
If $\Omega \subset \mathbb R^{n+1}$ is such that $\partial \Omega$ is Ahlfors regular and $\nu$ is a Carleson measure then we say that $\Omega$ is a {\bf (parabolic) uniformly rectifiable domain.}

If $\Omega$ is a parabolic uniformly rectifiable domain which is also parabolic NTA we say that $\Omega$ is a {\bf parabolic chord arc domain}. 

Finally, if $\nu$ satisfies a vanishing Carleson condition, that is if for any $K \subset \subset \mathbb R^{n+1}$ we have $$\lim_{\rho\downarrow 0} \sup_{(Q,\tau) \in K \cap \partial \Omega} \rho^{-(n+1)}\nu(\Delta_\rho(Q,\tau) \times [0,\rho]) = 0,$$ then we call $\Omega$ a {\bf parabolic vanishing chord arc domain}. \end{defin}

\begin{cor}\label{AhlforsRegularParabolicUR}
 Assume there is a function $f: \mathbb R \rightarrow \mathbb R$ such that  $\Omega = \{(x,t) \mid x > f(t)\}$. If $\sigma$ satisfies the lower Ahlfors regular condition, then $|f(t) - f(s)| \leq 7|t-s|^{1/2}$. If $\nu$ (as defined in \eqref{whatisnu}) is a Carleson measure, then $\|D^{1/2}_t f\|_{\mathrm{BMO}} < \infty$. 
 
 Similarly, if $\Omega = \{(x,t)\mid g(t) > x > f(t)\}$ and $\sigma$ is Ahlfors regular than both $g,f \in \Lip(1/2)$ with norm bounded by 14. Additionally, if $\nu$ is a Carleson measure than both $\|D^{1/2}_t f\|_{\mathrm{BMO}}, \|D^{1/2}_t g\|_{\mathrm{BMO}} < \infty$. 
 
 That is to say, a parabolic chord arc domain in $\mathbb R^2$ is actually a parabolic regular graph or slab domain. 
\end{cor}

\begin{proof}
We prove the theorem when $\Omega$ is a graph domain. The ``slab" case follows similarly.

Assume that $\sigma$ is Ahlfors regular, and, arguing to obtain a contradiction, assume that there are two times $t,s$ such that $|f(t) - f(s)| \geq 7|t-s|^{1/2}$. Without loss of generality, let $t = 0$ and $s = 1$ (as all the relevant conditions are scale and translation invariant).

 $f(t)$ is continuous, so there are points $t_i \in [0,1]$ such that $f(t_i) = i$ for $i = 1,2,...,7$. Therefore, $C_1(f(t_i), t_i) \cap C_1(f(t_j), t_j) = \emptyset$ if $i \neq j$. By lower Ahlfors regularity, $\sigma(C_1(f(t_i), t_i)) > 1/2$. On the other hand, $\sum_{i=1}^7 \sigma(C_1(f(t_i), t_i)) \leq 3$, as all the points in $C_1(f(t_i), t_i)$ have $t$ values between $-1$ and $2$ and there is no overlap between the cylinders. Summing up, we get $7/2 < 3$ a contradiction. 

Once we know that $f$ is Lipschitz, a harmonic analysis argument shows that if $\nu$ is a Carleson measure, then $D^{1/2}_tf\in \mathrm{BMO}(\mathbb R)$. This is proven at the end of Section 2 in \cite{hlncaloricmeasure}.
\end{proof}

\section{A Free Boundary Problem for the Parabolic Poisson Kernel}\label{secglobalfreeboundaryproblem}

Our main application of Theorem \ref{NTAdomainisagraphdomain} is a classification of blow-up solutions to a free boundary problem for the parabolic Poisson kernel in $\mathbb R^2$. To introduce this problem, let us recall the concept of the caloric Green function and caloric measure with a pole at infinity. 

If $\Omega$ is a parabolic NTA domain, we can define $\omega$, the caloric measure with a pole at infinity, and $u \in C(\Omega)$, the associated Green function, which satisfy  \begin{equation}\label{ipgreenfunction}\tag{IP} \left\{ \begin{aligned}
u(Y,s) \geq& 0, \;\forall (Y,s) \in \Omega,\\
u(Y,s) \equiv& 0,\; \forall (Y,s) \in \partial \Omega,\\
-(\partial_s + \Delta_Y)u(Y,s) =& 0, \; \forall (Y,s) \in \Omega\\
\int_{\partial \Omega} \varphi d\omega =& \int_\Omega u(Y,s) (\Delta_Y - \partial_s)\varphi dYds, \; \forall \varphi \in C^\infty_c(\mathbb R^{n+1}).
\end{aligned} \right. \end{equation} (For the existence, uniqueness and some properties of this measure/function, see Appendix C in \cite{engelsteincaloricmeasure}). There are also analogous objects for the adjoint equation. Note that we can define caloric measure and the caloric Green function with a ``finite pole" but that working with these is more complicated due to the anisotropic nature of the heat equation. To not get bogged down in technical details, we will work only with poles at infinity but all the following results hold for finite poles with just minor modifications. 

Let us now recall some salient concepts of ``regularity" for  $\omega$.

\begin{defin}\label{doubling}
We say $\omega$ is a {\bf doubling measure} if there exists a $c > 0$ such that $\omega(\Delta_{2r}(Q,\tau)) \leq c\omega(\Delta_r(Q,\tau))$ for all $r > 0$ and $(Q,\tau) \in \partial \Omega$. 
\end{defin}

\begin{defin}\label{apweight}
We say $\omega \in A_\infty(d\sigma)$ ({\bf is an $A_\infty$-weight}) if $\omega << \sigma$ on $\partial \Omega, h(Q,\tau) \colonequals \frac{d\omega}{d\sigma}$, and there exists a $c > 0$ such that \begin{equation}\label{reverseholderinequality2} \fint_{\Delta_{2r}(Q,\tau)} h(Q,\tau)^p d\sigma(Q,\tau) \leq c \left(\fint_{\Delta_r(Q,\tau)} h(Q,\tau)d\sigma(Q,\tau)\right)^p.\end{equation}
\end{defin}

Closely related to being an $A_\infty$-weight are the $\mathrm{BMO}$ and $\mathrm{VMO}$ function classes. 

\begin{defin}\label{bmo}
We say that $f\in \mathrm{BMO}(\partial \Omega)$ with norm $\|f\|_*$ if $$\sup_{r > 0} \sup_{(Q,\tau) \in \partial \Omega} \fint_{\Delta_r(Q,\tau)} |f(P,\eta) - f_{\Delta_r(Q,\tau)}| d\sigma(P,\eta)\leq \|f\|_*,$$ where $f_{\Delta_r(Q,\tau)} \equiv \fint_{\Delta_r(Q,\tau)} f(P,\eta) d\sigma(P,\eta)$, the average value of $f$ on $\Delta_r(Q,\tau)$. 

Define $\mathrm{VMO}(\partial \Omega)$ to be the closure of uniformly continuous functions vanishing at infinity in $\mathrm{BMO}(\partial \Omega)$. 
\end{defin}

We are broadly interested in the question of what the regularity of $\omega$ tells us about the regularity of $\Omega$. In the simplest case, we ask that $\omega \equiv \sigma$. We call these ``blowup" solutions because they appear as the limit of properly rescaled and translated domains which satisfy a more general free boundary problem (see Lemma \ref{pseudoblowupsareblowupsolutions} below). Theorem 1.10 in \cite{engelsteincaloricmeasure} states that a blowup solution, in any dimension, under an additional flatness assumption, must be  a half-space. The following result removes this flatness assumption in $\R^2$. 

\begin{thm}\label{globalblowupsinr2}
Let $\Omega \subset \mathbb R^2$ be a parabolic chord arc domain such that $\partial \Omega$ is connected and the caloric measure at infinity, $\omega$, satisfies $\omega \equiv \sigma$ (i.e. the parabolic Poisson kernel, $h \equiv \frac{d\omega}{d\sigma}$, satisfies $h \equiv 1$). Then, after a possible translation and reflection, $\Omega = \{(x,t)\mid x > 0\}$ and $u(x,t) = x$. 
\end{thm}

\begin{proof} Theorem 1.5 in \cite{nystromindiana} states that if $\omega = \sigma$ in a parabolic regular graph domain in $\R^2$, then that domain is a half-space and $u(x,t) = x$. Corollary \ref{AhlforsRegularParabolicUR} implies that a parabolic chord arc domain in $\R^2$ with connected boundary must be a parabolic regular graph domain, and thus the result follows. We should note that Theorem 1.5 in \cite{nystromindiana} is stated for parabolic regular graph domains with ``small constant" but that the assumption of smallness is not actually necessary (see the end of \cite{nystromzeitschrift} for more details). 
\end{proof}

\subsection{Blowups of Parabolic Chord Arc Domains in $\mathbb R^2$} Free boundary problems for the Poisson kernel have been studied in arbitrary dimensions under various additional flatness or smallness assumptions in \cite{hlncaloricmeasure}, \cite{nystromfenn}, \cite{nystromindiana}, \cite{nystromzeitschrift} and \cite{engelsteincaloricmeasure}. We will outline here how our classification of blowup solutions (and Theorem \ref{NTAdomainisagraphdomain}) renders these assumptions unnecessary in $\R^2$. 

\begin{thm}\label{vmotheorem}[Compare with Theorem 1.9 in \cite{engelsteincaloricmeasure}] Let $\Omega \subset \R^{2}$ be parabolic chord arc domain with $\log(h) \in \mathrm{VMO}(\partial \Omega)$, then $\Omega$ is a parabolic vanishing chord arc domain. 
\end{thm}

Given Theorem \ref{vmotheorem}, arguing as in Sections 6 and 7 of \cite{engelsteincaloricmeasure} shows that H\"older regularity on the part of the Poisson kernel implies H\"older regularity on the part of the boundary. 

\begin{thm}\label{holdertheorem}[Compare with Theorem 1.11 in \cite{engelsteincaloricmeasure}] Let $\Omega \subset \R^2$ be a parabolic chord arc domain with $\log(h) \in \mathbb C^{k+\alpha, (k+\alpha)/2}(\R^2)$ for $k \geq 0$ and $\alpha \in (0,1)$, then each connected component of $\partial \Omega$ is the graph of a $C^{(k+1+\alpha)/2}$ function in time. 

Furthermore, if $\log(h)$ is analytic in $X$ and in the second Gevrey class in $t$ then, each connected component of $\partial \Omega$ is graph of a function in the second Gevrey class in $t$. Similarly, if $\log(h) \in C^\infty$, then each component of $\partial \Omega$ is given by the graph of a $C^\infty$ function in time.
\end{thm}

(See \cite{engelsteincaloricmeasure} for more discussion on these theorems and the precise definition of the Gevrey class and the parabolic H\"older spaces). 

To prove Theorem \ref{vmotheorem} we must define pseudo-blowups (first introduced by Kenig and Toro, \cite{kenigtoroannsci}, to study the analogous elliptic problem). 

\begin{defin}\label{parabolicpseudoblowup}
Let $K$ be a compact set, $(Q_i,\tau_i) \in K \cap \partial \Omega$ and $r_i \downarrow 0$. Then we define 

\begin{subequations}
\label{eq:blowups}
\begin{align}
\Omega_i \colonequals& \{(x,t)\mid (r_ix + Q_i, r_i^2t + \tau_i)\in \Omega\} \\
u_i(x,t) \colonequals& \frac{u(r_ix + Q_i, r_i^2t+\tau_i)}{r_i \fint_{\Delta_{r_i}(Q_i,\tau_i)} h d\sigma} \label{ublowup}\\
\omega_i(E) \colonequals& \frac{\sigma(\Delta_{r_i}(Q_i,\tau_i))}{r_i^{n+1}} \frac{\omega(\{(P,\eta)\in \Omega\mid ((P-Q_i)/r_i, (\eta-\tau_i)/r_i^2)\in E\})}{\omega(\Delta_{r_i}(Q_i,\tau_i))}\label{omegablowup}\\
\sigma_i \colonequals& \sigma|_{\partial \Omega_i} \label{sigmablowup}
    \end{align}
\end{subequations}
\end{defin}

The following Lemma tells us that the pseudo-blowups of parabolic chord arc domains are parabolic chord arc domains in $\R^2$. 

\begin{lem}\label{blowupsarechordarcdomains}
Let $\Omega \subset \mathbb R^2$ be a parabolic chord arc domain, $K$ be compact, $(Q_i,\tau_i) \in K\cap \partial \Omega$ and $r_i\downarrow 0$. Then (possibly passing to a subsequence) there is some $\Omega_\infty \subset \mathbb R^2$ such that $\Omega_i \rightarrow \Omega_\infty$ in the Hausdorff distance sense. Furthermore, $\Omega_\infty$ is a parabolic regular graph domain. Finally, $\omega_i \rightharpoonup \omega_\infty$ and $u_i \rightarrow u_\infty$ (uniformly on compact sets), where $\omega_\infty$ and $u_\infty$ are the caloric measure at infinity and Green function at infinity for $\partial \Omega_\infty$. 
\end{lem}

\begin{proof}
By passing to a subsequence, we can assume that $(Q_i, \tau_i)$ are all contained in one connected component of $\partial \Omega$. By Theorem \ref{NTAdomainisagraphdomain}, this component is given by the graph of some function $f$. Then $\partial \Omega_\infty$ is the graph of $f_\infty$ which is the uniform limit of the functions $f_i$, where $f_i(t) =\frac{f(r_i^2t + \tau_i) -f(\tau_i)}{r_i}$. Note that the $f_i$ are pre-compact in the $\Lip(1/2)$ space and that the claimed regularity of $f_\infty$ then follows from standard arguments.  

That $\{u_i\}$ and $\{\omega_i\}$ are pre-compact follows from standard estimates on $\Lip(1/2)$ graph domains (see, e.g., \cite{lewisandmurray}, Chapter 3, Section 6). Of course, $\Omega$ need not be a graph domain, but for any compact set $K$, $\partial \Omega_i\cap K$ will be given by a single graph if $i$ is large enough. It is then easy to see that the relevant estimates (i.e. the boundary Harnack inequality) hold in this case. For more details see Section 4 in \cite{engelsteincaloricmeasure} or Lemmas 16 and 17 in \cite{nystromfenn}.

Finally, that $u_\infty$ and $\omega_\infty$ are the Green function and caloric measure follows immediately from the fact that $u_i \rightarrow u_\infty$ and $\omega_i \rightharpoonup \omega_\infty$. 
\end{proof}

When $\log(h) \in \mathrm{VMO}$, we can use a harmonic analysis argument to show that each pseudo-blowup must be a ``blowup" solution (i.e. that $\omega_\infty = \sigma_\infty$). 

\begin{lem}\label{pseudoblowupsareblowupsolutions}
Let $\Omega, h$ satisfy the hypothesis of Theorem \ref{vmotheorem}. Let $K$ be compact and fix two sequences, $(Q_i,\tau_i) \in K\cap \partial \Omega$ and $r_i\downarrow 0$.  Let $\Omega_\infty, \omega_\infty$ and $u_\infty$ be the parabolic chord arc domain, caloric measure at infinity and Green function at infinity given by Lemma \ref{blowupsarechordarcdomains}. Then $\omega_\infty = \sigma_\infty$ ($\sigma_\infty$ is the parabolic surface measure supported on $\partial \Omega_\infty$). 
\end{lem}

\begin{proof}
That $\sigma_i \rightharpoonup \omega_\infty$ follows from a harmonic analysis argument using the condition, $\log(h) \in \mathrm{VMO}(\partial \Omega)$. This argument is given in detail in the proof of Lemma 3.3 in \cite{nystromindiana}, and in Lemma 4.5 and Proposition 4.7 in \cite{engelsteincaloricmeasure}.

 In $\R^2$ we have, $$\lim_{i\rightarrow \infty} \int_{\partial \Omega_i} \varphi d\sigma_i = \lim_{i \rightarrow \infty} \int_{\partial \Omega_i} \varphi dt = \int_{\partial \Omega_\infty} \varphi dt = \int_{\partial \Omega_\infty} \varphi d\sigma_\infty,\; \forall \varphi \in C_c^\infty(\R^2).$$  The first equality above is due to the fact that, for large enough $i$, $\mathrm{spt}\; \varphi\cap \partial \Omega_i$ is a graph (on which parabolic surface measure is simply $dt$). 

Since $\sigma_i \rightharpoonup \omega_\infty$ and $\sigma_i \rightharpoonup \sigma_\infty$ we must conclude that $\sigma_\infty = \omega_\infty$. 
\end{proof}

Lemma \ref{pseudoblowupsareblowupsolutions} combined with Theorem \ref{globalblowupsinr2} implies that if $\Omega, \log(h)$ are as in Theorem \ref{vmotheorem}, then every pseudo-blowup of $\Omega$ is a half-plane. From here the conclusion that $\Omega$ is a parabolic vanishing chord arc domain follows easily from the dominated convergence theorem (for more details see \cite{nystromzeitschrift}). Actually, a more general fact is true. Let $\Omega$ be a parabolic chord arc domain in any dimension with the property that for any pseudo-blowup $\Omega_i \rightarrow \Omega_\infty$, $\Omega_\infty$ is a half-space and $\sigma_i\rightharpoonup \sigma_\infty$. Then $\Omega$ is a parabolic vanishing chord arc domain (this is Proposition 5.1 in \cite{engelsteincaloricmeasure}). Theorem \ref{vmotheorem} follows.

\bibliography{ParabolicBiblio}{}
\bibliographystyle{amsalpha}

\end{document}